\documentclass[12pt]{article}
\usepackage{latexsym}
\usepackage{amssymb, amsthm}
\usepackage{amsmath}
\usepackage{makeidx}
\usepackage{undertilde}
\usepackage{enumerate}
\usepackage{verbatim}

\newtheorem{theorem}{Theorem}[section]
\newtheorem{proposition}[theorem]{Proposition}
\newtheorem{definition}[theorem]{Definition}

\newtheorem{lemma}[theorem]{Lemma}
\newtheorem{corollary}[theorem]{Corollary}

\def\c{{\mathbf{C}}}
\def\bbC{{\mathbf{C}}}

\def\cH{{\mathcal{H}}}

\newcommand{\bbP}{\mathbb{P}}

\newcommand{\rank}{\mathrm{rank}}

\newcommand{\ZFC}{\mathsf{ZFC}}

\newcommand{\card}[1]{{\vert #1 \vert} }

\newcommand{\forces}{\Vdash}

\renewcommand{\models}{\vDash}

\newcommand{\bR}{{\mathbb{R}}}

\newcommand{\powerset}{\mathcal{P}}

\newcommand{\cf}{{\rm cf}}

\def\k{\kappa}
\def\a{\alpha}
\def\b{\beta}

\def\l{\lambda}

\newcommand{\rprop}[1]{Proposition~\ref{#1}}

\newcommand{\rlem}[1]{Lemma~\ref{#1}}

\newcommand{\AD}{\mathsf{AD}}
\newcommand{\ADR}{\mathsf{AD}_{\mathbb{R}}}

\newcommand{\rmw}{\mathrm{w}}
\newcommand{\Add}{\mathrm{Add}}
\newcommand{\pmax}{\mathbb{P}_{\mathrm{max}}}

\newcommand{\bbQ}{\mathbb{Q}}
\newcommand{\bbR}{\mathbb{R}}

\newcommand{\breals}{\omega^{\omega}}
\newcommand{\HOD}{\mathrm{HOD}}
\newcommand{\DC}{\mathsf{DC}}

\newcommand{\cP}{\mathcal{P}}

\newcommand{\MM}{\mathsf{MM}}

\newcommand{\ZF}{\mathsf{ZF}}

\newcommand{\less}{\mathord{<}}
\newcommand{\cof}{\mathrm{cof}}
\newcommand{\Col}{\mathrm{Col}}
\newcommand{\restrict}{\mathord{\upharpoonright}}

\newcommand{\PFA}{\mathsf{PFA}}

\title{Failures of square in $\pmax$ extensions of Chang models
}

\author{Paul B. Larson and Grigor Sargsyan}

\date{\today}

\makeindex

\begin{document}

\maketitle

\begin{abstract}
We show that the statements $\square(\omega_{3})$ and $\square(\omega_{4})$ both fail in the $\pmax$ extension of a variation of the Chang model introduced by Sargsyan.
\end{abstract}



\section{Introduction}

The Chang model \cite{Chang} is the smallest inner model of Zermelo-Fraenkel set theory ($\ZF$) containing every countable sequence of ordinals. Variations of the Chang model can be produced by adding parameters, restricting to the countable sequences from some fixed ordinal or by closing under ordinal definability. In this paper we show that (consistently, assuming the consistency of certain large cardinals) Jensen's square principles $\square(\omega_{3})$ and $\square(\omega_{4})$ fail in extensions of certain Chang models by Woodin's $\pmax$ forcing. The existence of Chang models with the required properties is proved in \cite{Changmodels_1} by the second author, from the existence of a Woodin cardinal which is a limit of Woodin cardinals.

In combination with the results of \cite{Changmodels_1}, the results in this paper have consequences for the inner model theory program.
In particular, we produce forcing extensions of Chang models satisfying the assumptions of the first sentence of the following theorem (see Theorems \ref{weakmain} and \ref{mainthrm}).

\begin{theorem}[Jensen-Schimmerling-Schindler-Steel \cite{JSSS}]\label{JSSSthrm} Assume that $\aleph_2^\omega=\aleph_2$ and that the principles $\square(\omega_3)$ and $\square_{\omega_3}$ both fail to hold. Let $g\subseteq \Col(\omega_3, \omega_3)$ be a $V$-generic filter. If $V[g]\models ``\sf{K}^c_{jsss}$ converges" then $(\sf{K}^c_{jsss})^{V[g]}\models$ ``there is a subcompact cardinal".
\end{theorem}

Since subcompact cardinals have greater consistency strength than Woodin limits of Woodin cardinals, this gives the following theorem, where the transitive model is the model $V[g]$ from Theorem \ref{JSSSthrm}.

\begin{theorem}\label{Kc corollary} It is consistent relative to the existence of a Woodin cardinal that is a limit of Woodin cardinals that there is a transitive model of $\sf{ZFC}$ in which the $\sf{K}^c_{jsss}$ construction does not converge.
\end{theorem}

This paper does not involve any inner model theory. We refer the reader to \cite{Changmodels_1} for a discussion of $\sf{K}^c_{jsss}$ and Theorem \ref{Kc corollary}.

\subsection{Square principles}

The square principles we consider in this paper were introduced by Ronald Jensen \cite{Jensen}. We briefly review their definitions.

\begin{definition}
Given a cardinal $\kappa$, the principle $\square_{\kappa}$ says that there exists a sequence $\langle C_{\alpha} : \alpha < \kappa^{+} \rangle$ such that for each $\alpha < \kappa^{+}$,
\begin{itemize}
\item each $C_{\alpha}$ is a closed cofinal subset of $\alpha$;
\item for each limit point $\beta$ of $C_{\alpha}$, $C_{\beta} = C_{\alpha} \cap \beta$;
\item the ordertype of each $C_\alpha$ is at most $\kappa$.
\end{itemize}
\end{definition}

For any cardinal $\kappa$, $\square_{\kappa}$ implies the statement $\square(\kappa^{+})$ as defined below.

\begin{definition}
Given an ordinal $\gamma$, the principle $\square(\gamma)$ says that there exists a sequence $\langle C_{\alpha} : \alpha < \gamma \rangle$ such that
\begin{itemize}
\item for each $\alpha < \gamma$,
\begin{itemize}
\item each $C_{\alpha}$ is a closed cofinal subset of $\alpha$;
\item for each limit point $\beta$ of $C_{\alpha}$, $C_{\beta} = C_{\alpha} \cap \beta$;
\end{itemize}
\item there is no thread through the sequence, i.e., there is no closed unbounded $E \subseteq \gamma$ such that $C_{\alpha} = E \cap \alpha$ for every limit point $\alpha$ of $E$.
\end{itemize}
\end{definition}

A $\square(\gamma)$-\emph{sequence} is a sequence $\langle C_{\alpha} : \alpha < \gamma \rangle$ as in the definition of $\square(\gamma)$.
A \emph{potential} $\square(\gamma)$-sequence is a sequence $\langle C_{\alpha} : \alpha < \gamma \rangle$ satisfying all but the last condition in the definition.  An elementary argument gives the important fact that if $\gamma$ has uncountable cofinality, then each potential $\square(\gamma)$-sequence has at most one thread.

We will in fact obtain the negation of a weaker version of square, also due to Jensen.

\begin{definition} Given an ordinal $\gamma$ and a cardinal $\delta$, the principle $\square(\gamma, \delta)$ asserts
the existence of a sequence \[\langle \mathcal{C}_{\alpha} \mid \alpha < \gamma \rangle\] satisfying the following conditions.
\begin{itemize}
\item
For each $\alpha < \gamma$,
\begin{itemize}
\item
$0 < |\mathcal{C}_{\alpha}| \leq \delta$;
\item
each element of $\mathcal{C}_{\alpha}$ is club in $\alpha$;
\item
for each member $C$ of $\mathcal{C}_{\alpha}$, and each limit point $\beta$ of $C$, \[C \cap
\beta \in\mathcal{C}_{\beta}.\]
\end{itemize}
\item
There is no thread through the sequence, that is, there is no club $E \subseteq
\gamma$ such that $E \cap \alpha \in \mathcal{C}_{\alpha}$ for every limit point $\alpha$ of $E$.
\end{itemize}
\end{definition}

As above, a $\square(\gamma,\delta)$-sequence is a sequence $\langle C_{\alpha} : \alpha < \gamma \rangle$ as in the definition of $\square(\gamma, \delta)$. A \emph{potential} $\square(\gamma,\delta)$-sequence is a sequence $\langle C_{\alpha} : \alpha < \gamma \rangle$ satisfying all but the last condition in the definition. Again, an elementary argument shows that if the cofinality of $\gamma$ is greater than $|\delta|^{+}$, then each potential $\square(\gamma, \delta)$-sequence has at most $|\delta|$ many threads.
Note that $\square(\gamma)$ is $\square(\gamma, 1)$ and if $\delta < \eta$ then $\square(\kappa, \delta)$ implies $\square(\kappa, \eta)$.

We use Todorcevic's theorem \cite{Stevo_square} that if $\gamma$ has cofinality at least $\omega_{2}$ then the restriction of the Proper Forcing Axiom ($\PFA$) to partial orders of cardinality $\gamma^{\omega}$ implies the failure of $\square(\gamma, \omega_{1})$. For $\gamma < \omega_3$ this fragment of $\PFA$ follows from $\MM^{++}(\mathfrak{c})$ (a technical strengthening of the restriction of Martin's Maximum to partial orders of cardinality at most the continuum), since $\MM^{++}(\mathfrak{c})$ implies that $\mathfrak{c} = \aleph_{2}$ by the results of \cite{FMSI}. We will not need the definition of $\MM^{++}(\mathfrak{c})$ in this paper, as our only use of it will be to apply Todocevic's theorem, and Woodin's theorems on obtaining $\MM^{++}(\mathfrak{c})$ in $\pmax$ extensions (see Subsection \ref{pmaxssec}).

Theorem \ref{weakmain} is one version of the main theorem of this paper. A more explicit version is given in Theorem \ref{mainthrm} below. In light of Todorcevic's theorem it should be possible to replace $\neg\square(\omega_{3}, \omega)$  and $\neg\square(\omega_{4}, \omega)$ below with $\neg\square(\omega_{3}, \omega_1)$ and $\neg\square(\omega_{4}, \omega_1)$, but this remains open.

\begin{theorem}\label{weakmain} The consistency of $\ZFC$ plus the existence of a Woodin limit of Woodin cardinals implies the consistency of \[\ZFC + \aleph_2^\omega=\aleph_2 + \neg\square(\omega_{3}, \omega)  + \neg\square(\omega_{4}, \omega).\]
\end{theorem}

\subsection{Chang models and $\Join_{\lambda}$}\label{cmjoinssec}

We let $\cH$ represent the class of pairs of ordinals $(\alpha, \beta)$ such that the $\alpha$th element of the standard definability order of $\HOD$ is an element of the $\beta$th. This is just a technical convenience that allows us to give a concise statement of results from \cite{Changmodels_1}; the only property of $\cH$ we use in this paper is that it is a definable class of pairs of ordinals. Given an ordinal $\gamma$, we write $\cH\restrict \gamma$ for $\cH \cap (\gamma \times \gamma)$. Given an ordinal $\gamma$,
we write $\c_\gamma^{-}$ for the structure $L_\gamma(\cH, \gamma^{\omega})$, which is constructed relative to the predicate $\cH$, adding (for each ordinal $\alpha < \gamma$) all $\omega$-sequences from $\alpha$ at stage $\alpha + 1$. Note that $\gamma$ is the ordinal height of this structure. We also write $\c_\gamma$ for $L(\cH\restrict \gamma, \gamma^{\omega})$ and $\c^+_\gamma$ for $\HOD_{\gamma^{\omega}}$.


We let $\rmw(A)$ denote the Wadge rank of a set $A \subseteq \breals$, and for any ordinal $\alpha$ let $\Delta_{\alpha}$ denote the set of subsets of $\breals$ of Wadge rank less than $\alpha$.
We will work with models of $\AD^{+}$ (an extension of the Axiom of Determinacy due to Hugh Woodin; see for instance \cite{LarsonAD}) in which some ordinal satisfies the following statement (we refer the reader to \cite{Solovay, Woodin, LarsonAD} for the definition of the Solovay sequence).

\begin{definition}
For an ordinal $\lambda$, $\Join_\lambda$ is the statement that, letting $\kappa$ be $\Theta^{\c_\lambda}$,
\begin{itemize}
\item $\kappa$ is a regular member of the Solovay sequence below $\Theta$,
\item $\c^+_{\lambda}\models \lambda=\kappa^{+} +\cf(\lambda)=\lambda$,
\item $\c_\lambda^{-} \cap \powerset(\bR)
=\c_{\lambda}^+\cap \powerset(\bR)=\Delta_\kappa$,
\item $\powerset(\kappa^\omega)\cap \c_{\lambda}=\powerset(\kappa^\omega)\cap \c_{\lambda}^+$,
\item $\kappa\leq \cf(\lambda)$.
\end{itemize}
\end{definition}


Since $\AD^{+}$ implies that successor members of the Solovay sequence below $\Theta$ have cofinality $\omega$ (see \cite{LarsonAD}), $\AD^{+}$ + $\Join_{\lambda}$ implies that $\kappa = \Theta^{\c_\lambda}$ is a limit member of the Solovay sequence. Woodin has shown that $\AD^{+}$ implies each of the following (see \cite{LarsonAD}):
\begin{itemize}
\item $\AD^{+}$ holds in every inner model of $\ZF$ containing $\bbR$;
\item $\ADR$ holds if and only if the Solovay sequence has limit length.
\end{itemize}
It follows that, assuming $\Join_{\kappa^{+}}$, $L(\Delta_{\kappa}) \models \ADR$.


We let $\ddagger$ stand for the theory $\sf{ZF}$ + $V=L(\powerset(\bR))$ + $\sf{AD}_{\mathbb{R}}$ + ``$\Theta$ is regular".
Results of Solovay from \cite{Solovay} say that $\ddagger$ implies $\DC$ (the statement that every tree of height $\omega$ without terminal nodes has a cofinal branch) and also the statement that the sharp of each set of reals exists. By
results of Becker and Woodin (see \cite{LarsonAD}), $\ADR + \DC$ implies that all subsets of $\breals$ are Suslin, and thus that $\AD^{+}$ holds.





Models of $\exists \lambda \Join_{\lambda}$ are given by the following theorem from \cite{Changmodels_1}.

\begin{theorem}\label{inputthrm}
  Suppose that there exists a Woodin cardial which is a limit of Woodin cardinals. Then in a forcing extension there is an inner model satisfying $\ddagger$  + $\exists \lambda \Join_{\lambda}$.
\end{theorem}


\subsection{Variants of $\DC$}

The principle of Dependent Choice ($\DC$) can be varied by restricting the nodes of the tree to some set, or by considering trees of uncountable height.

Given a binary relation $R$ on a set $X$ and an ordinal $\delta$ we say that $f\colon\delta\rightarrow X$ is an $R$-\emph{chain} if $f(\alpha)R f(\beta)$ holds for all ordinals $\alpha < \beta$ below $\delta$. Given an ordinal $\eta$ we say that $R$ is $\eta$-\emph{closed} if for every $\delta<\eta$ and for every $R$-chain $f\colon \delta \rightarrow X$ there is an $r\in X$ such that for every $\alpha < \delta$, $f(\alpha) R r$. We then say that $\sf{DC}_{\gamma}$ holds for a cardinal $\gamma$ if for every  cardinal $\eta\leq \gamma$ and every $\eta$-closed binary relation $R$ there is an $R$-chain $f:\gamma \rightarrow X$. We write $\DC$ for $\DC_{\omega}$.

Given a set $X$ and a cardinal $\gamma$, we write $\DC_{\gamma}\restrict X$ for the restriction of $\DC_{\gamma}$ to binary relations on $X$, which we also call $\DC_{\gamma}$ \emph{for relations on} $X$.

\subsection{$\pmax$}\label{pmaxssec}

The partial order $\pmax$ was introduced by Woodin in \cite{Woodin}. We list here the facts about $\pmax$ (all from \cite{Woodin}) that we will need.
\begin{itemize}
\item $\pmax$ conditions are elements of $H(\aleph_{1})$ and the corresponding order is definable in $H(\aleph_{1})$.

\item $\pmax$ is $\sigma$-closed.

\item  Forcing with $\pmax$ over a model of $\AD^{+} + \DC$ preserves the property of having cofinality at least $\omega_{2}$ (this follows from a combination of Theorems 3.45 and 9.32 of \cite{Woodin}, as outlined in Section \ref{threadsec} below).

\item If $M$ is a model of $\ZF + \AD^{+}$ and $G\subseteq \pmax^{M}$ is an $M$-generic filter, then the following hold in $M[G]$:
\begin{itemize}
\item $2^{\aleph_{0}} = \aleph_{2}$; \item $\Theta^{M} = \omega_{3}$; \item $\cP(\omega_{1}) \subseteq L(\bbR)[G]$.
\end{itemize}



\item Forcing with $\pmax$ over a model of $\ADR$ +  $V = L(\cP(\bbR))$ + ``$\Theta$ is regular" produces a model of $\ZF$ + $\DC_{\aleph_{2}}$ +  $\MM^{++}(\mathfrak{c})$.
\end{itemize}

Forcing with $\pmax$ over a model of $\ADR$ cannot wellorder $\cP(\bbR)$ (since a name for such a wellorder would induce a failure of Uniformization), but $\DC_{\aleph_{2}} + 2^{\aleph_{0}} = \aleph_{2}$ implies that $\cP(\bbR)$ may be wellordered  by forcing with $\Add(\omega_{3}, 1)$ (where, for any ordinal $\gamma$, $\Add(\gamma, 1)$ is the partial order adding a generic subset of $\gamma$ by initial segments). Since (by $\DC_{\aleph_{2}}$) $\Add(\omega_{3}, 1)$ does not add subsets of $\omega_{2}$, $\MM^{++}(\mathfrak{c})$ is preserved. This gives the following theorem, which is essentially Theorem 9.39 of \cite{Woodin}.

\begin{theorem}[Woodin]\label{wsthrm} Forcing with $\pmax * \Add(\omega_{3}, 1)$ over a model of $\ADR$ +  $V = L(\cP(\bbR))$ + ``$\Theta$ is regular" produces a model of $\ZFC$ + $\MM^{++}(\mathfrak{c})$.
\end{theorem}

Again, it follows from Todorcevic's theorem that $\square(\omega_{2}, \omega_{1})$ fails in such an extension.

By the results mentioned at the end of Section \ref{cmjoinssec}, the following hold in the context of $\ddagger$ + $\Join_{\lambda}$:
\begin{itemize}
\item $\ADR$ +  $V = L(\cP(\bbR))$ + ``$\Theta$ is regular";
\item the sharp of each subset of $\bbR$ exists;
\item $\bbC^{+}_{\lambda}$ $\models$ $\ADR$ + ``$\Theta$ is regular".
\end{itemize} 
However, $\bbC^{+}_{\lambda}$ is not a model of ``$V =L(\cP(\bbR))$", since, being closed under ordinal definability, it contains the sharp of its version of $\cP(\bbR)$ (i.e., $\Delta_{\kappa}$).
So we cannot just cite Theorem \ref{wsthrm} for our main result. As we shall see, it suffices, however, to wellorder $\lambda^{\omega}$, which can be done by forcing with $\Add(\omega_{4}, 1)$.


The following then is our main theorem.
The theorem builds upon \cite{SquarePaper} and, of course, \cite{Woodin}. As we shall see in Section \ref{threadsec}, the proof uses an argument from the proof of \cite[Theorem 7.3]{SquarePaper}.

\begin{theorem}\label{mainthrm}
Suppose $V\models \ddagger$ and that $\lambda$ is an ordinal for which $\Join_{\lambda}$ holds. Let $\kappa = \Theta^{\c_{\lambda}}$.
Let $(G, H, K)$ be a $V$-generic filter for the forcing iteration
\[(\pmax*\Add(\k, 1)*\Add(\l, 1))^{\c^{+}_{\l}}.\]
Then
\[\c^+_{\l}[G, H, K]\models \ZFC + \sf{MM}^{++}(c)+\neg\square(\omega_3, \omega)+\neg\square(\omega_{4}, \omega).\]
\end{theorem}

For the rest of the paper we fix $\kappa$, $\lambda$, $G$, $H$ and $K$ as in the statement of Theorem \ref{mainthrm}.
Since $\pmax \subseteq H(\aleph_{1})$, and $\kappa$ is both regular and equal to $\Theta^{\c^{+}_{\lambda}}$, \[(\pmax * \Add(\kappa, 1))^{L(\Delta_{\k})}\] is the same as
$(\pmax * \Add(\kappa, 1))^{\c^{+}_{\lambda}}$.
In addition the partial orders \[(\pmax*\Add(\k, 1)*\Add(\l, 1))^{\c_{\l}}\] and $(\pmax*\Add(\k, 1)*\Add(\l, 1))^{\c^{+}_{\l}}$ are the same, from which it follows that the theorem implies the corresponding version with $\bbC_{\lambda}$ in place of $\bbC^{+}_{\lambda}$.




\section{Threading coherent sequences}\label{threadsec}

The material in this section is adapted from \cite{SquarePaper}, and reduces (via Theorem \ref{theorem:vanilla}) the proof of Theorem \ref{mainthrm} to showing the following:
\begin{itemize}
  \item $\Add(\kappa, 1)*\Add(\lambda, 1)$ is $\omega_{2}$-closed in $\c^{+}_{\lambda}[G]$;
  \item $V[G] \models ((\Add(\kappa, 1)*\Add(\lambda, 1))^{\c^{+}_{\lambda}[G]})^{\omega_{1}} \subseteq \c^{+}_{\lambda}[G]$;
  \item $\c^{+}_{\lambda}[G, H] \models \DC_{\aleph_{3}}$.
\end{itemize}
The first of these follows from the fact that $\c^{+}_{\lambda}[G] \models \DC_{\aleph_{2}}$, which is shown in Lemma \ref{cpgdc2lem}.
The second is Lemma \ref{kaplamcllem}.
The third is Lemma \ref{dcomega_3lem}.
The first two facts show that the partial order $(\Add(\kappa, 1)*\Add(\lambda, 1))^{\c^{+}_{\lambda}[G]}$ satisfies, in $V[G]$, the conditions on the partial order $\bbQ$ from the statement of Theorem \ref{theorem:vanilla}.
This gives the failures of $\square(\omega_{3}, \omega)$ and $\square(\omega_{4}, \omega)$ in $\c^{+}_{\lambda}[G,H,K]$. The third
is used only to show that $\c^{+}_{\lambda}[G,H,K]$ is a model of $\ZFC$.


In order to apply Todorcevic's theorem to show that $\square(\omega_{3}, \omega)$ and $\square(\omega_{4}, \omega)$ fail, we need to show that $\kappa$ and $\lambda$ (from Theorem \ref{mainthrm}) have cofinality $\omega_{2}$ in
$V[G]$ (recall that they are less than $\Theta^{V}$, which is $\omega_{3}^{V[G]}$). To do this, we use the following covering theorem of Woodin from Section 3.1 of \cite{Woodin}.
The notion of {\em $A$-iterability} in the following theorem is introduced in
Woodin \cite[Definition 3.30]{Woodin}. Given $X\prec H(\omega_2)$, $M_X$ denotes its transitive
collapse.

\begin{theorem}[Woodin {\cite[Theorem 3.45]{Woodin}}] \label{thm:woodin}
Suppose that $M$ is a proper class inner model that contains all the reals and satisfies $\AD+\DC$.
Suppose that for any $A\in{\mathcal P}(\bbR)\cap M$, the set
$$ \{X\prec H(\omega_2)\mid \mbox{\rm$X$ is countable, and $M_X$ is $A$-iterable}\} $$
is stationary. Let $X$ in $V$ be a bounded subset of $\Theta^M$ of size $\omega_1$. Then there is a set
$Y\in M$, of size $\aleph_1$ in $M$, such that $X\subseteq Y$.
\end{theorem}

We apply Theorem \ref{thm:woodin} in the proof of Lemma \ref{lemma:coflemma} with $M$ as a model of the form $L(A,\bbR)$ for some $A \subseteq \breals$, and the $V$ of Theorem \ref{thm:woodin} as a $\pmax$ extension of $M$.



\begin{lemma} \label{lemma:coflemma}
Suppose that $M$ is a model of $\ZF + \AD^{+}$ and $\gamma$ is an ordinal of cofinality at least
$\omega_{2}$ in $M$.
Let $G_0\subset \pmax$ be an $M$-generic filter. Then $\gamma$ has cofinality at least $\omega_{2}$ in
$M[G_0]$.
\end{lemma}

\begin{proof}
Suppose first that $\gamma < \Theta^{M}$. Let $X$ be a subset of $\gamma$ of cardinality $\aleph_{1}$ in $M[G_0]$,
and let $A \in \cP(\breals) \cap M$ have Wadge rank at least $\gamma$. Since $|\gamma| \leq 2^{\aleph_{0}}$ in $M[G_{0}]$ and $\mathcal{P}(\omega_{1})^{M[G_0]}$ is contained in $L(A, \mathbb{R})[G]$ by Theorem 9.23 of \cite{Woodin},
$X$ is in $L(A, \mathbb{R})[G]$. By Theorem 9.32 of Woodin, the hypotheses of Theorem \ref{thm:woodin} are satisfied with $L(A, \bbR)$ as $M$ and $L(A, \bbR)[G]$ as $V$.
Applying Theorem \ref{thm:woodin} we have that $X$ is a subset of an element
of $L(A, \mathbb{R})$ of cardinality $\aleph_{1}$ in $L(A, \mathbb{R})$.

The lemma follows immediately from the previous paragraph for $\gamma$ of cofinality less than $\Theta^{M}$ in $M$.
If $\gamma \geq \Theta$ is regular in $M$ there is no cofinal function from $\breals$ to $\gamma$ in $M$, so there is no such
function in $M[G_0]$, either.
The theorem then follows for arbitrary $\gamma$.
\end{proof}

In conjunction with the facts mentioned at the beginning of this section, the following theorem (with $M_{1}$ as $V$, $M_{0}$ as $\bbC^{+}_{\lambda}$, $\gamma$ as either $\kappa$ or $\lambda$ and $\bbQ$ as $(\Add(\kappa,1) * \Add(\lambda,1))^{\bbC^{+}_{\lambda}[G]}$) completes the proof of Theorem \ref{mainthrm}.
The theorem and its proof are taken from \cite{SquarePaper}, except that the specific partial order used in \cite{SquarePaper} has been replaced with a more general class of partial orders.

\begin{theorem} \label{theorem:vanilla}
Suppose that $M_{1}$ is a model of $\ddagger$, and that for some set $X \in M_{1}$ containing
$\breals \cap M_{1}$, $M_{0} = \HOD^{M_{1}}_{X}$. Suppose also that  $\Theta^{M_{0}} < \Theta^{M_{1}}$ and that $\gamma \in [\Theta^{M_{0}}, \Theta^{M_{1}})$ has cofinality at least $\omega_{2}$ in $M_{1}$.
Let $G_0 \subset \pmax$ be $M_{1}$-generic, and let $I \subset \bbQ$ be
$M_{1}[G_0]$-generic, for some partial order $\bbQ \in M_{0}[G_0]$ which, in $M_{1}[G_0]$, is
$\less\omega_{2}$-directed closed and of cardinality at most $\mathfrak{c}$.
Then $\square(\gamma, \omega)$ fails in $M_{0}[G_0][I]$.
\end{theorem}

\begin{proof}
Suppose that $\tau$ is a $\pmax * \dot{\bbQ}$-name in $M_{0}$ for a
$\square(\gamma,\omega)$-sequence. We may assume that the realization of $\tau$ comes with
an indexing of each member of the sequence in order type at most $\omega$. In $M_{0}$, $\tau$ is
ordinal definable from some $S \in X$.

By Theorems 9.35 and 9.39 of \cite{Woodin}, $\DC_{\aleph_{2}}$ and $\MM^{++}(\mathfrak{c})$ hold in $M_{1}[G_0]$.
By Lemma \ref{lemma:coflemma},
$\gamma$ has cofinality $\omega_{2}$ in $M_{1}[G_0]$.
Forcing with $\less\omega_{2}$-directed closed partial orders
of size at most $\mathfrak{c}$ preserves $\MM^{++}(\mathfrak{c})$ (see\cite{Larson:separating}). It follows then
that $\DC_{\aleph_{1}}$ and $\MM^{++} (\mathfrak{c})$ hold in the $\dot{\bbQ}_{G_0}$-extension of $M_{1}[G_0]$,
and thus that in this extension every potential $\square(\gamma, \omega)$-sequence is
threaded.

Let $\mathcal{C} = \langle \mathcal{C}_{\alpha} : \alpha < \gamma \rangle$ be the realization of
$\tau$ in the $\dot{\bbQ}_{G}$-extension of $M_{1}[G_0]$. Since $\gamma$ has cofinality at least $\omega_{2}$ in this
extension, which satisfies $\DC_{\aleph_{1}}$, $\mathcal{C}$ has at most $\omega$ many
threads, since otherwise one could find a $\mathcal{C}_{\alpha}$ in the sequence with uncountably
many members. Therefore, some member of some $\mathcal{C}_{\alpha}$ in the realization of $\tau$
will be extended by a unique thread through the sequence, and since the realization of $\tau$ indexes
each $\mathcal{C}_{\alpha}$ in order type at most $\omega$, there is in $M_{1}$ a name, ordinal
definable from $S$, for a thread through the realization of $\tau$. This name is then a member of
$M_{0} = \HOD^{M_{1}}_{X}$.
\end{proof}

\section{Proving $\DC_{\aleph_{m}}$}

As stated at the beginning of Section \ref{threadsec}, two of our three remaining tasks are showing that $\c^{+}_{\lambda}[G] \models \DC_{\aleph_{2}}$ and $\c^{+}_{\lambda}[G,H] \models \DC_{\aleph_{3}}$. Section \ref{dcreducesec} reduces each of these to the case of relations on $\lambda^{\omega}$. In Section \ref{dcproveoutlinessec} we outline our strategy for proving that $\DC_{\aleph_{2}}$ holds in $\bbC^{+}_{\lambda}[G]$ for relations on $\lambda^{\omega}$. A proof of $\c^{+}_{\lambda}[G,H] \models \DC_{\aleph_{3}}$ (using essentially the same strategy) is given in Section \ref{dc3sec}.

\subsection{Reducing to $\DC_{\aleph_{m}}\restrict \lambda^{\omega}$}\label{dcreducesec}

Lemma \ref{dcreducelem} is applied in this paper in the cases $m=2$ and $m=3$ (recall that, as we have defined it, $\DC_{\aleph_{m}}$ implies $\DC_{\aleph_{k}}$ for all $k \leq m$). Since (by the theorem of Solovay cited in Subsection \ref{cmjoinssec}), $\ddagger$ implies $\DC$, the lemma also shows (in the case $m=0$) that $\DC$ holds in $\bbC^{+}_{\lambda}$.


\begin{lemma}\label{dcreducelem} Let $\bbP$ be a partial order in $\bbC^{+}_{\lambda}$, and let $I \subseteq \bbP$ be a $\bbC^{+}_{\lambda}$-generic filter. Let $m$ be an element of $\omega$ such that $\DC_{\aleph_{k}}$ holds in $\bbC^{+}_{\lambda}[I]$ for all $k < m$.
Suppose also that, in $\bbC^{+}_{\lambda}[I]$,  every $\less\omega_{m}$-closed tree  on $\lambda^{\omega}$ of height $\omega_{m}$ has a cofinal branch.
Then $\DC_{\aleph_{m}}$ holds in $\bbC^{+}_{\lambda}[I]$.
\end{lemma}

\begin{proof}
Fix a $\less\omega_{m}$-closed tree $T$ in $\bbC^{+}_{\lambda}[I]$. Fix an ordinal $\gamma$ such that every node of $T$ is the realization of a $\bbP$-name which is ordinal definable in $V_{\gamma}$ from some element of $\lambda^{\omega}$.
Given $(n,\delta, x) \in \omega \times \gamma \times \lambda^{\omega}$, let \[t_{n,\delta, x}\] be the set defined in $V_{\gamma}$ from $\delta$ and $x$ by the formula with G\"{o}del number $n$.

Let $T'$ be the tree of sequences $\langle x_{\alpha} : \alpha < \beta\rangle$ (for some $\beta < \omega_{m}$) for which there exists a sequence \[\langle y_{\alpha} : \alpha < \beta \rangle\] such that, for each $\eta < \beta$, \[y_{\eta} = t_{n,\delta, x_{\eta}, I},\] where $(n, \delta) \in (\omega, \gamma)$ is minimal such that  $t_{n,\delta,x_{\eta}}$ is a $\bbP$-name and
\[\langle y_{\alpha} : \alpha  < \eta \rangle^{\frown} \langle t_{n,\delta,x_{\eta},I}\rangle \in T.\]

Then $T'$ is also $\less\omega_{m}$-closed, and an $\omega_{m}$-chain through $T'$ induces one through $T$.
\end{proof}

\subsection{Proving $\DC_{\aleph_{2}}| \lambda^{\omega}$}\label{dcproveoutlinessec}

To show that $\DC_{\aleph_{2}} \restrict \lambda^{\omega}$ holds in $\bbC^{+}_{\lambda}[G]$, we show that the following statements hold in $\bbC^{+}_{\lambda}[G]$:

\begin{itemize}
\item there is no cofinal map from $\omega_{2}$ to $\lambda$;

\item there is no cofinal map from $(\gamma^{\omega})^{\beta}$ to $\lambda$, for any $\gamma < \lambda$ and $\beta < \omega_{2}$ (it suffices to show this for $\gamma = \kappa$ and $\beta = \omega_{1}$).
\end{itemize}

The first of these follows from Lemma \ref{strongreglem} with $b$ as $\omega_{2} \times \pmax$. The second is shown in the proof of Lemma \ref{cpgdc2lem}, whose statement is just the desired statement that $\bbC^{+}_{\lambda}[G] \models \DC_{\aleph_{2}}$. These two facts imply that every cardinal $\delta \leq \aleph_{2}$ and each $\delta$-closed relation $R$ on $\lambda^{\omega}$ in $\bbC^{+}_{\lambda}[G]$, there exists a $\gamma < \lambda$ such that $R \cap \gamma^{\omega}$ is also $\delta$-closed. Since $\lambda = \kappa^{+}$, it suffices then (once we have established the two facts above) to consider trees on $\kappa^{\omega}$.
To show that, in $\bbC^{+}_{\lambda}[G]$, every $\omega_{2}$-closed tree on $\kappa^{\omega}$ has a cofinal branch, we use the fact (which follows from standard $\pmax$ arguments) that the following statements hold in $\bbC^{+}_{\lambda}[G]$:

\begin{itemize}
\item there is no cofinal map from $\omega_{2}$ to $\kappa$ (because $\kappa = \Theta^{\bbC^{+}_{\lambda}}$ is regular in $\bbC^{+}_{\lambda}$ and $\pmax \subseteq H(\aleph_{1})$);

\item there is no cofinal map from $(\gamma^{\omega})^{\beta}$ to $\kappa$, for any $\gamma < \kappa$ and $\beta < \omega_{2}$ (because $\kappa = \omega_{3}^{\bbC^{+}_{\lambda}[G]}$ and $\aleph_{2}^{\aleph_{1}} = \aleph_{2}$).
\end{itemize}
These facts imply that it suffices to consider $\omega_{2}$-closed trees on $\gamma^{\omega}$ for any $\gamma < \kappa$.
Since each such $\gamma^{\omega}$ is a surjective image of the wellordered set $\breals$ in $\bbC^{+}_{\lambda}[G]$,
$\bbC^{+}_{\lambda}[G]$ satisfies the statement that each such tree has a cofinal branch.


\section{Strong regularity of $\lambda$}

In this section we prove a regularity property of $\lambda$ in $\bbC^{+}_{\lambda}$ and derive several consequences, including the fact that $\DC_{\aleph_{2}}$ holds in $\bbC^{+}_{\lambda}[G]$. We will refer to the property of $\lambda$ established in Lemma \ref{strongreglem} as \emph{strong regularity}.

\begin{lemma}\label{strongreglem} Whenever
$b\in \bbC^-_{\lambda}$ and $f \colon b \to \lambda$ is in $\bbC_{\lambda}^+$,  there exists a $\gamma<\lambda$ such that $f[b]\subseteq \gamma$.
\end{lemma}

\begin{proof} Let $\beta<\lambda$ be such that $b\in \bbC^{-}_{\beta}$.
Since
$\lambda = \kappa^{+}$ in $\bbC_{\lambda}$ there exists a surjection $h\colon \kappa\to \beta$ in $\bbC_{\lambda}$ (recall that $\bbC^{+}_{\lambda}$ and $\bbC_{\lambda}$ have the same subsets of $\kappa$).
Let $B$ be the set of $y \in \kappa^\omega$ such that $b$ has a member definable in $\bbC^{-}_{\beta}$ from
$h \circ y$ and $\cH \restrict \beta$.
Then $B$ induces a surjection $g \colon \kappa^{\omega} \to b$ in $\bbC_{\lambda}$.

Since $\kappa = \Theta^{\bbC^{+}_{\lambda}}$ is regular, the ordertype of $f[g[\alpha^{\omega}]]$ less than $\kappa$ for each $\alpha < \kappa$.
Since $\lambda$ is regular in $\bbC^{+}_{\lambda}$, $f[g[\alpha^{\omega}]]$ is a bounded subset of $\lambda$, for each $\alpha < \kappa$.
Again applying the regularity of $\lambda$ in $\bbC^{+}_{\lambda}$, $f[b]$ is bounded in $\lambda$.

\end{proof}

It follows immediately from Lemma \ref{strongreglem} that there is no cofinal map from $\omega_{2}$ to $\lambda$ in $\bbC^{+}_{\lambda}[G]$.
As noted in Section \ref{dcproveoutlinessec}, this reduces our proof that $\bbC^{+}_{\lambda}[G] \models \DC_{\aleph_{2}}$ to showing that there is no
cofinal map from $\kappa^{\omega_{1}}$ to $\lambda$ in $\bbC^{+}_{\lambda}[G]$. We will in fact show that $\kappa^{\omega_{1}} \cap \bbC^{+}_{\lambda}[G] \in \bbC^{-}_{\lambda}[G]$, which will suffice, by Lemma \ref{strongregghlem} below, which shows that the strong regularity of $\lambda$ persists to $\bbC^{+}_{\lambda}[G, H]$.
Lemmas \ref{cmpscllem} and \ref{dkcontlem} use the strong regularity of $\lambda$ to prove closure properties of $\bbC^{-}_{\lambda}$.
The proof of Lemma \ref{cmpscllem} is similar to the proof of Lemma \ref{strongreglem}.

\begin{lemma}\label{cmpscllem}
For all $b\in \bbC^-_{\lambda}$, $\cP(b) \cap \bbC^{+}_{\lambda} \subseteq \bbC^{-}_{\lambda}$.
\end{lemma}

\begin{proof}
Fix $b \in \bbC^{-}_{\lambda}$ and a $\beta<\lambda$ such that
$b\in \bbC^{-}_{\beta}$. Let
$h \colon \kappa \to \beta $ be a surjection in $\bbC_{\lambda}$.
Fix $a\in \powerset(b)\cap \bbC^+_{\lambda}$  and let $B_{a}$ be the set of $(x,n) \in \kappa^\omega \times \omega$ such that some member of $a$ is definable over $\bbC^{-}_{\beta}$ from $h \circ x$ and $\cH \restrict \beta$ via the formula with G\"{o}del number $n$.
Then $B_{a} \in \bbC^{+}_{\lambda}$. Since $\cP(\kappa^{\omega}) \cap \bbC^{+}_{\lambda} \subseteq \bbC_{\lambda}$, $B_{a}\in \bbC_{\lambda}$, so $a\in \bbC_{\lambda}$. A reflection argument using the strong regularity of $\lambda$ shows that $a \in \bbC^{-}_{\lambda}$.
\end{proof}

The following lemma implies that $(\pmax * \Add(\kappa, 1))^{\bbC^{+}_{\lambda}} \in \bbC^{-}_{\lambda}$ and
$(\pmax * \Add(\kappa, 1) * \Add(\lambda, 1))^{\bbC^{+}_{\lambda}} \subseteq \bbC^{-}_{\lambda}$.

\begin{lemma}\label{dkcontlem} $\Delta_{\kappa} \in \bbC^{-}_{\lambda}$
\end{lemma}

\begin{proof} For each $\alpha < \kappa$ let $f(\alpha)$ be the least $\beta < \lambda$ such that there is a set of reals of Wadge rank $\alpha$ in $\bbC^{-}_{\beta}$. By Lemma \ref{cmpscllem}, $f$ is well-defined.
By the strong regularity of $\lambda$ in $\bbC^{+}_{\lambda}$, the range of $f$ is bounded below $\lambda$.
It follows then that $\Delta_{\kappa} \in \bbC^{-}_{\eta}$ for any $\eta < \lambda$ containing the range of $f$.
\end{proof}

Since $(\pmax * \Add(\kappa, 1))^{\bbC^{+}_{\lambda}} \in \bbC^{-}_{\lambda}$, we get the following lemma, which implies that $\lambda$ is regular in $\bbC^{+}_{\lambda}[G,H]$.

\begin{lemma}\label{strongregghlem} Whenever
$b\in \bbC^-_{\lambda}[G,H]$ and
\[f \colon b \to \lambda\] is in $\bbC_{\lambda}^{+}[G,H]$,  there exists a $\gamma<\lambda$ such that $f[b]\subseteq \gamma$.
\end{lemma}




\begin{lemma}\label{cpgdc2lem} $\bbC^{+}_{\lambda} \models \DC_{\aleph_{2}}$
\end{lemma}

\begin{proof} By the remarks in Section \ref{dcproveoutlinessec}, and Lemma \ref{strongreglem}, it suffices to show that there is no cofinal map
from $\kappa^{\omega_{1}}$ to $\lambda$ in $\bbC^{+}_{\lambda}[G]$. Since $\kappa$ is regular in $V$, $\cof(\kappa) > \omega_{1}$ in $\bbC^{+}_{\lambda}[G]$, so every element of $\kappa^{\omega_{1}} \cap \bbC^{+}_{\lambda}[G]$ is the realization of a name coded by a set of reals.
Since \[\Delta_{\kappa} = \cP(\bbR) \cap \bbC^{+}_{\lambda} \in \bbC^{-}_{\lambda}\] by Lemma \ref{dkcontlem}, $(\kappa^{\omega_{1}})^{\bbC^{+}_{\lambda}[G]} \in \bbC^{-}_{\lambda}[G]$. A reference to Lemma \ref{strongregghlem} then completes the proof.
\end{proof}







\section{$\omega_{1}$-closure in $V[G]$}





In this section we show that, in $V[G]$, $\bbC^{+}_{\lambda}[G]$ is closed under $\lambda$-sequences from
$(\Add(\kappa, 1) * \Add(\lambda, 1))^{\bbC^{+}_{\lambda}[G]}$, which is the second statement from the beginning of Section \ref{threadsec}. This Lemma \ref{kaplamcllem} below, which follows from Lemma \ref{omega1 functions}.

\begin{lemma}\label{omega1 functions} In $V[G]$, for each $b\in \bbC^-_{\lambda}[G]$,
$b^{\omega_1}\in \bbC^{-}_{\lambda}[G]$.
\end{lemma}

\begin{proof}
Since $\lambda < \Theta$, $\bbC^{-}_{\lambda}$ is a surjective image of $\breals$ in $V$.
Let $U \subseteq \breals$ be such that $\bbC^{-}_{\lambda}$ is a surjective image of $\breals$ in $L(U, \bbR)$.
Since $\breals$ is wellordered in $L(U, \bbR)[G]$, there exists in $L(U, \bbR)[G]$ a function picking for each
$x \in \bbC^{-}_{\lambda}[G]$ a $\pmax$-name $\tau_{x} \in \bbC^{-}_{\lambda}$ such that $\tau_{x,G} = x$.
Since $\cP(\omega_{1}) \cap V[G] \subseteq L(\bbR)[G]$, \[(\bbC^{-}_{\lambda}[G])^{\omega_{1}} \cap V[G]
\subseteq L(U, \bbR)[G].\]


We work in $L(U, \bbR)[G]$, which satisfies Choice.
Fix $b\in \bbC^-_{\lambda}[G]$, and let $\beta < \lambda$ be
such that $\Delta_{\kappa}, \tau_{b} \in \bbC^{-}_{\beta}$.
It follows that every member of $b$ is the realization of a name in $\bbC^{-}_{\beta}$.
We first show that $b^{\omega_1}\subseteq \bbC^{-}_{\lambda}[G]$.

Fix $f \in b^{\omega_{1}}$. Since Choice holds, there is an \[h_f \in (\bbC^{-}_{\beta})^{\omega_{1}}\] such that, for every $\alpha<\omega_1$,
$h_f(\alpha)$ is a $\pmax$ name in $\bbC^{-}_{\beta}$ such that $h_{f}(\alpha)_{G} = f(\alpha)$.
Fix  a function $c_{f} \colon \omega_{1} \to \omega$  and a sequence $\langle B_\alpha:\alpha<\omega_1\rangle$ such that each
$B_\alpha$ is a nonempty subset of $\beta^{\omega}$ and each $h_f(\alpha)$ is definable in
$\bbC^{-}_{\beta}$ from $\cH \restrict \beta$, and each member of the corresponding $B_\alpha$,
via the formula with G\"{o}del number $c_{f}(\alpha)$.
Since $\cP(\omega_{1}) \subseteq L_{\lambda}(\bbR)[G]$, $c_{f} \in \bbC_{\lambda}[G]$.

Let $h \colon \kappa \to \beta$ be a surjection in $\bbC^{-}_\lambda$ (which exists by Lemma \ref{cmpscllem}), and, for each $\alpha < \omega_1$ let
\[B'_\alpha = \{ x \in \kappa^{\omega} : h \circ x \in B_\alpha\}.\]
As there is no cofinal function from $\omega_{1}$ to $\kappa$ in $\bbC^{+}_{\lambda}[G]$, there is a $\gamma < \kappa$ such that $B'_\alpha \cap \gamma^{\omega}$ is nonempty
 for each $\alpha < \omega_{1}$.
Let $r \colon \breals \to \gamma^{\omega}$ be a surjection in $L(\Delta_{\kappa})$, and for each $\alpha < \omega_{1}$ let
\[C_\alpha = r^{-1}[B'_a \cap \gamma^{\omega}].\]
Then each $C_\alpha$ is a set of reals in $L(\Delta_{\kappa})$.

In $L(\Delta_{\kappa})$ there is a set of reals of Wadge rank greater than each $C_{\alpha}$, so, in $L(U, \bbR)[G]$,  there is a subset $T$ of $\omega_{1}$ such that
\[\langle C_\alpha : \alpha < \omega_{1} \rangle \in L(\Delta_\kappa)[T].\]
Since $\powerset(\omega_{1}) \subseteq L_{\lambda}(\bbR)[G]$,
it follows that $\langle B'_{\alpha} \cap \gamma^{\omega} : \alpha \in \omega_{1} \rangle$ is in $L(\Delta_{\kappa})[G]$, and that
\[\langle \{ h \circ x : x \in B'_{\alpha} \cap \gamma^{\omega}\} : \alpha \in \omega_{1} \rangle,\]
$\langle h_{f}(\alpha) : \alpha < \omega_{1} \rangle$ and $f$  are in $\bbC^{-}_{\lambda}[G]$.

Suppose now that $b^{\omega_1}\not \subseteq \bbC_{\alpha}[G]$ for any $\alpha < \lambda$.
We then have a function $g\colon b^{\omega_1}\to \lambda$ that is unbounded in $\lambda$ with
$g\in \bbC_{\lambda}[G]$.
Using the above coding, $g$ induces a cofinal function
\[h \colon \powerset(\omega_{1}) \times \Delta_{\kappa} \to \lambda\] in $\bbC_{\lambda}[G]$, with the first argument playing the role of $T$ above and the second coding both a wellordering of $\breals$ in ordertype $\gamma$ and set of reals of Wadge rank above each $C_{\alpha}$.
This contradicts Lemma \ref{strongregghlem}.
\end{proof}

\begin{lemma}\label{kaplamcllem} $V[G] \models ((\Add(\kappa, 1)*\Add(\lambda, 1))^{\bbC^{+}_{\lambda}[G]})^{\omega_{1}} \subseteq \bbC^{+}_{\lambda}[G]$
\end{lemma}

\begin{proof}
As noted before Lemma \ref{dkcontlem}, each element of $((\Add(\kappa, 1)*\Add(\lambda, 1))^{\bbC^{+}_{\lambda}[G]}$
an element of $\bbC^{-}_{\lambda}[G]$.
Since $\cof(\lambda) = \omega_{2}$ in $V[G]$, every element of
\[((\Add(\kappa, 1)*\Add(\lambda, 1))^{\bbC^{+}_{\lambda}[G]})^{\omega_{1}}\] in $V[G]$ has range contained in some element of $\bbC^{-}_{\lambda}[G]$. The lemma then follows from Lemma \ref{omega1 functions}.
\end{proof}

\section{$\sf{DC}_{\aleph_3}$ in $\bbC^+_{\lambda}[G,H]$}\label{dc3sec}




Lemma \ref{dcomega_3lem} is the third item from the beginning of Section \ref{threadsec}, and completes the proof of Theorem \ref{mainthrm}.
The proof is a reflection argument as in Subsection \ref{dcproveoutlinessec}.

\begin{lemma}\label{omega2 closure}
There are stationarily many $\eta < \lambda$ such that, in $\bbC^+_{\lambda}[G, H]$,
\[\bbC^{-}_{\eta}[G, H]^{\omega_2}\subseteq \bbC^{-}_{\eta}[G, H].\]
\end{lemma}

\begin{proof}

Since $\lambda$ is regular in $\bbC^{+}[G, H]$ by Lemma \ref{strongregghlem}, it suffices to show that for all $\alpha < \beta < \lambda$, if there exists a surjection $s \colon \kappa \to \alpha$ in $\bbC^{-}_{\beta}$ then \[\bbC^{-}_{\alpha}[G, H]^{\omega_2}\subseteq \bbC^{-}_{\beta}[G, H].\]
Fix such $\alpha < \beta$, and let $s \colon \kappa \to \alpha$ be a surjection in $\bbC^{-}_{\beta}$.
Fix a function $f \colon \omega_{2} \to \bbC^{-}_{\alpha}[G, H]$ in $\bbC^{+}_{\lambda}[G,H]$.

For each $\gamma < \omega_{2}$, let $B_{\gamma}$ be the set of $x \in \alpha^{\omega}$ such that
$f(\gamma)$ is definable in $\bbC^{-}_{\alpha}[G, H]$ from $\cH \restrict \alpha$, $G$, $H$ and $x$ via the formula with G\"{o}del code $x(0)$.
Since $\cof(\kappa) > \omega_{2}$ in $\bbC^{+}_{\lambda}[G, H]$, there is a $\delta < \kappa$ such that for all $\gamma < \omega_{2}$, \[B'_{\gamma} = \{ y \in \delta^{\omega} : s \circ y \in B_{\gamma}\}\] is nonempty.
The sequence $\langle B'_{\gamma} : \gamma < \omega_{2} \rangle$ is coded by a set of reals in $\bbC^{+}_{\lambda}[G, H]$, so it is in $L(\Delta_{\kappa})[G]$. It follows that $f \in \bbC^{-}_{\beta}[G,H]$.
\end{proof}

\begin{lemma}\label{dcomega_3lem} $\bbC^+_{\lambda}[G, H]\models \sf{DC}_{\aleph_3}$.
\end{lemma}

\begin{proof}
By Lemma \ref{dcreducelem}, it suffices to show that, in $\bbC^{+}_{\lambda}[G, H]$, every $\less\omega_{3}$-closed tree of height $\omega_{3}$ on $\lambda^{\omega}$ has a cofinal branch.

In $\bbC^{+}_{\lambda}[G,H]$, $\cof(\lambda) > \omega_{3}$ and $\lambda$ is strongly regular, by Lemma \ref{strongregghlem}.
By Lemma \ref{omega2 closure}, then, it suffices to consider trees on $\kappa^{\omega}$.
Since $\kappa^{\omega}$ is wellordered in $\bbC^{+}_{\lambda}[G, H]$ the lemma follows.
\end{proof}

\section{Further work}

The arguments in this paper naturally adapt to produce models of $\ZFC$ in which $\square(\aleph_{n}, \omega)$ fails for all $n \in \omega$, from models of the appropriate generalizations of $\Join_{\lambda}$. Since these generalizations are not yet known to be consistent, we save these arguments for a later paper. In addition, there is much more that can be said about the types of Chang models that we consider in this paper. Some observations that were not needed for the proof of Theorem \ref{mainthrm} have been collected in \cite{LSmoreChang}.

\bibliographystyle{plain}
\bibliography{KC}

\begin{thebibliography}{10}

\bibitem{SquarePaper}
Andr\'{e}s~Eduardo Caicedo, Paul Larson, Grigor Sargsyan, Ralf Schindler, John
  Steel, and Martin Zeman.
\newblock Square principles in {${\Bbb P}_{\max}$} extensions.
\newblock {\em Israel J. Math.}, 217(1):231--261, 2017.

\bibitem{Chang}
C.~C. Chang.
\newblock Sets constructible using {$L_{\kappa \kappa }$}.
\newblock In {\em Axiomatic {S}et {T}heory ({P}roc. {S}ympos. {P}ure {M}ath.,
  {V}ol. {XIII}, {P}art {I}, {U}niv. {C}alifornia, {L}os {A}ngeles, {C}alif.,
  1967)}, pages 1--8. Amer. Math. Soc., Providence, R.I., 1971.

\bibitem{FMSI}
M.~Foreman, M.~Magidor, and S.~Shelah.
\newblock Martin's maximum, saturated ideals, and nonregular ultrafilters. {I}.
\newblock {\em Ann. of Math. (2)}, 127(1):1--47, 1988.

\bibitem{Jensen}
R.~Bj\"{o}rn Jensen.
\newblock The fine structure of the constructible hierarchy.
\newblock {\em Ann. Math. Logic}, 4:229--308; erratum, ibid. 4 (1972), 443,
  1972.
\newblock With a section by Jack Silver.

\bibitem{JSSS}
Ronald Jensen, Ernest Schimmerling, Ralf Schindler, and John Steel.
\newblock Stacking mice.
\newblock {\em The Journal of Symbolic Logic}, 74(01):315--335, 2009.

\bibitem{LarsonAD}
Paul Larson.
\newblock {\em Extensions of {A}xioms of {D}eterminacy}.
\newblock In preparation, available http://www.users.miamioh.edu/larsonpb/.

\bibitem{Larson:separating}
Paul Larson.
\newblock Separating stationary reflection principles.
\newblock {\em J. Symbolic Logic}, 65(1):247--258, 2000.

\bibitem{LSmoreChang}
Paul Larson and Grigor Sargsyan.
\newblock More on {C}hang models.
\newblock In preparation.

\bibitem{Changmodels_1}
Grigor Sargsyan.
\newblock Chang models over universally baire sets and the iterability problem
  for countable submodels of $\sf{K}^{c}$.
\newblock preprint.

\bibitem{Solovay}
Robert~M. Solovay.
\newblock The independence of {DC} from {AD}.
\newblock In {\em Cabal {S}eminar 76--77 ({P}roc. {C}altech-{UCLA} {L}ogic
  {S}em., 1976--77)}, volume 689 of {\em Lecture Notes in Math.}, pages
  171--183. Springer, Berlin, 1978.

\bibitem{Stevo_square}
Stevo Todor\v{c}evi\'{c}.
\newblock A note on the proper forcing axiom.
\newblock In {\em Axiomatic set theory ({B}oulder, {C}olo., 1983)}, volume~31
  of {\em Contemp. Math.}, pages 209--218. Amer. Math. Soc., Providence, RI,
  1984.

\bibitem{Woodin}
W.~Hugh Woodin.
\newblock {\em The axiom of determinacy, forcing axioms, and the nonstationary
  ideal}, volume~1 of {\em De Gruyter Series in Logic and its Applications}.
\newblock Walter de Gruyter GmbH \& Co. KG, Berlin, revised edition, 2010.

\end{thebibliography}
\end{document}